\documentclass[reqno]{amsart}
\usepackage[margin=3cm]{geometry}

\usepackage{amssymb}
\usepackage[mathcal]{euscript}
\usepackage{relsize}

\let\oldtocsection=\tocsection 
\let\oldtocsubsection=\tocsubsection 
\renewcommand{\tocsection}[2]{\hspace{0mm}\oldtocsection{#1}{#2}}
\renewcommand{\tocsubsection}[2]{\hspace{1em}\oldtocsubsection{#1}{#2}}

\usepackage{hyperref}
\hypersetup{
    colorlinks,
    citecolor=black,
    filecolor=black,
    linkcolor=black,
    urlcolor=black
}

\usepackage{comment}
\DeclareFontFamily{OT1}{pzc}{}
\DeclareFontShape{OT1}{pzc}{m}{it}{<-> s * [1.100] pzcmi7t}{}
\DeclareMathAlphabet{\mathpzc}{OT1}{pzc}{m}{it}

   \usepackage{etoolbox}
    \patchcmd{\section}{\scshape}{\large\bfseries}{}{}
    \makeatletter
    \renewcommand{\@secnumfont}{\bfseries}
    \makeatother

\usepackage{quiver}
\usepackage{tikz-cd}
\tikzcdset{arrow style=math font}
\tikzcdset{scale cd/.style={every label/.append style={scale=#1},
    cells={nodes={scale=#1}}}}

\usepackage{biblatex}
\usepackage{booktabs}
\bibliography{references}

\numberwithin{equation}{section}
\newtheorem{theorem}{Theorem}[section]
\newtheorem*{theorem*}{Theorem}
\newtheorem{corollary}[theorem]{Corollary}
\newtheorem*{corollary*}{Corollary}
\newtheorem{lemma}[theorem]{Lemma}
\newtheorem*{lemma*}{Lemma}

\theoremstyle{definition}

\newtheorem{construction}{Construction}

\newtheorem{definition}[theorem]{Definition}
\newtheorem{remark}[theorem]{Remark}

\def\ZZ{\mathbb{Z}}

\newcommand{\sslash}{\mathbin{/\mkern-6mu/}}

\newcommand{\TopMon}{\textsf{TopMon}}

\newcommand{\Deltainj}{\Delta_{\text{inj}}}

\newcommand{\id}{\textsf{id}}
\newcommand{\Top}{\textsf{Top}}
\newcommand{\Gpd}{\textsf{Gpd}}
\newcommand{\B}{\textsf{B}}
\newcommand{\sSet}{\textsf{sSet}}
\newcommand{\sTop}{\textsf{sTop}}
\newcommand{\Set}{\textsf{Set}}
\newcommand{\hocolim}{\textsf{hocolim}}

\newcommand{\mm}{\mathcal{M}}
\newcommand{\uu}{\mathcal{U}}
\newcommand{\cc}{\mathcal{C}}
\newcommand{\OO}{\mathcal{O}}

\usepackage{float}
\usepackage{subcaption}
\begin{document}

\title[The operad associated to a crossed simplicial group
]{The operad associated to a crossed simplicial group
}

\author{A. Semidetnov}
\email{artemsemidetnov@gmail.com, Artem.Semidetnov@etu.unige.ch}
\address{Universit\'{e} de Gen\`{e}ve, Section de Math\'{e}matiques, Route de Drize 7, Villa Battelle, 1227 Carouge, Switzerland}

\begin{abstract}
    We introduce and study structured enhancement of the notion of a crossed simplicial group, which we call an operadic crossed simplicial group. We show that with each operadic crossed simplicial group one can associate a certain operad in groupoids. We demonstrate that symmetric and braid crossed simplicial groups can be made into operadic crossed simplicial groups in a natural way. For these two examples, we show that our construction of the associated operad recovers the $E_\infty$-operad and the $E_2$-operad respectively. We demonstrate the utility of this framework through two main applications: a generalized bar construction that specializes to Fiedorowicz's symmetric and braided bar constructions, and an identification of the associated group-completed monads with Baratt-Priddy-Quillen type spaces. 
\end{abstract}

\maketitle


\section*{Introduction}

Crossed simplicial groups were introduced by Fiedorowicz and Loday in ~\cite{f0e21457-c8c3-31f9-bd48-762cf5f699ae} and independently by Krasauskas in ~\cite{krasauskas_skew-simplicial_1987}. We use notations established in ~\cite{f0e21457-c8c3-31f9-bd48-762cf5f699ae}. 

A~ crossed simplicial group is a generalization of a simplicial group.
In short, it is a simplicial set $G_*$ such that each $G_n$ is a group that acts on the set $[n]$ of vertices of the $n$-simplex, and the face and degeneracy maps are crossed homomorphism with respect to the specified action. The latter means that for every elements \(g_1, g_2 \in G_n\) there are identities
\[
d_i(g_1g_2) = d_i(g_1)d_{g_1^{-1}i}(g_2), \quad s_i(g_1g_2) = s_i(g_1)s_{g_1^{-1}i}(g_2).
\]
Equivalently, a crossed simplicial group can be given as an extension \( \Delta G \) of the simplicial category \( \Delta \), where every morphism \( [m] \to [n] \) in \( \Delta G \) factors uniquely as an automorphism of \( [m] \) given by the action of an element of \( G_m \), followed by a morphism from \( \Delta \). 
The structural decomposition theorem for crossed simplicial groups is established in \cite{f0e21457-c8c3-31f9-bd48-762cf5f699ae} and \cite{krasauskas_skew-simplicial_1987} independently: for a crossed simplicial group $G_*$ there is the canonical epimorphism $\pi \colon G \twoheadrightarrow N_*$ to one of the 7 specific crossed simplicial groups with the kernel being a genuine simplicial group.
We rely on this theorem heavily and call the morphism $\pi$ the \textit{structural projection} for $G_*$.

A~key application of this theory, explored in a preprint by Fiedorowicz ~\cite{Fied-preprint}, is a generalization of the bar construction. For the crossed simplicial groups $S_*$ (built from symmetric groups), $B_*$ (built from Artin braid groups), he introduces the symmetric bar construction and the braided bar construction:
$$\B^S \colon \TopMon \to \Top^{\Delta S} \text{,} \qquad \B^B \colon \TopMon \to \Top^{\Delta B}.$$ Motivated by homological considerations, he proves that for a topological monoid $M$ there are weak equivalences 
$$\hocolim_{\Delta S} (\B^S M) \simeq \B(T_\infty, T_1, M), \qquad \hocolim_{\Delta B} (\B^B M) \simeq \B(T_2, T_1, M),$$
where $T_k$ denotes the monad associated to the little $k$-cubes operad, and $\B(-, -, -)$ denotes the monadic bar construction.

\subsection*{Overview of results}
We present a construction for the classifying space of a crossed simplicial group that is compatible with the structural decomposition.
Let \(G_*\) be a crossed simplicial group and \(\pi\) be its structural projection.
We construct a model for the classifying space of the simplicial group \(P_* = \mathrm{ker}(\pi)\) equipped with an action of the crossed simplicial group \(N_* = \mathrm{cod}(\pi)\). 
Namely, we prove the following theorem. 
\begin{theorem*}[Theorem~\ref{thm:groupoid_construction}]
Let \( G_* \) be a crossed simplicial group with the structural decomposition
\[
P_* \hookrightarrow G_* \overset{\pi}{\twoheadrightarrow} N_*.
\]
Then there exists a functor \( \Gamma_* \colon (\Delta N)^{\text{op}} \to \Gpd \) such that \( N_n \) acts freely on \( \Gamma_n \) and there is a homotopy equivalence \( \B\Gamma_n \simeq K(P_n, 1) \) for each \( n \).
\end{theorem*}

Next, we focus on the class of \emph{symmetric crossed simplicial groups}, namely those for which the structural quotient \( N_* \) is the symmetric crossed simplicial group \( S_* \). For such crossed simplicial groups, we introduce two structural enchancements: a \textit{monoidal crossed simplicial group} and an \textit{operadic crossed simplicial group} (see Definition~\ref{def:monoidal_csg} and Definition~\ref{def:operadic_csg}). These structures permit the construction of an operad built on the groupoids \(\Gamma_*\) from the theorem above.
Namely, we prove the following theorem.

\begin{theorem*}[Theorem~\ref{thm:set_operad}, Theorem~\ref{thm:gpd_operad}]
Let \( G_* \) be a monoidal crossed simplicial group. Then the sequence of sets \( \{G_n\}_{n \ge 0} \) admits an operad structure.
Moreover, if \( G_* \) is an operadic crossed simplicial group, there is a symmetric operad structure on the sequence of groupoids \(\{\Gamma_n\}_{n \ge 0}\).
\end{theorem*}
For \( G_* = B_* \) (Artin braid groups), the resulting operad is equivalent to the \( E_2 \)-operad, while for \( G_* = S_* \) (symmetric groups), it is equivalent to the \( E_\infty \)-operad; see Section~\ref{sec:examples} and Remark~\ref{rem:E_ooE_2}.

To highlight the motivation behind these theorems, we now list several of their immediate consequences. Let \( G_* \) be an operadic crossed simplicial group, and let \( T_G \) denote the monad associated to the operad \( \B\Gamma_* \) in \( \Top \). Recall that for a topological monoid there is a classifying space construction given by the composition of functors
\[\B \colon \TopMon \xrightarrow{N} \sTop \xrightarrow{|-|} \Top.\]

\begin{corollary*}
The group completion \( T_G(*)^{\mathrm{grp}} \) satisfies a weak equivalence
\[
T_G(*)^{\mathrm{grp}} \xrightarrow{\simeq} 
\Omega \B \left(\bigsqcup_{n \ge 0} \B G_n\right),
\]
where the symbol $\B$ on the right-hand side stands for the classifying space construction of a topological monoid.
\end{corollary*}
This corollary may be viewed as a manifestation of the Cohen--Segal--May and Barratt--Priddy--Quillen theorems. In the case $G = B_*$, it asserts that
\(
\Omega \B\left(\bigsqcup_{n \ge 0} \B B_n\right)
\)
is the free group-like $E_2$-algebra generated by a point, and thus it is weakly equivalent to $\Omega^2S^2$. Likewise, when $G = S_*$, the corollary identifies
\(
\Omega \B\left(\bigsqcup_{n \ge 0} \B S_n\right)
\)
as the free group-like $E_\infty$-algebra generated by a point and thus it is weakly equivalent to $\Omega^\infty \Sigma^\infty S^0$.

For any topological monoid \( M \), one can define its \( G_* \)-bar construction \[ \B^{G_*} M \colon \Delta G \to \Top,\]
generalizing the symmetric and braided bar constructions.
We obtain the following generalization of \cite[Proposition 1.4]{Fied-preprint}.
\begin{corollary*}
There is a weak equivalence
\[
\hocolim_{\Delta G} \B^{G_*} M \simeq \B(T_G, C_1, M),
\]
where the symbol $\B$ on the right-hand side stands for the monadic bar construction.
\end{corollary*}


\subsection*{Organization of the paper.}
In Section~\ref{sec:functorial_classifying_space}, we develop the constructions and proofs underlying Theorem~\ref{thm:groupoid_construction} for a crossed simplicial group. Section~\ref{sec:operadic_crossed_simplicial_groups} introduces the aforementioned enhancements of the structure of a crossed simplicial group and provides the relevant examples. In Section~\ref{sec:constructing_operads}, we review the basic notions of operad theory and construct operads in the categories $\Set$ and $\Gpd$, proving Theorem~\ref{thm:set_operad} and Theorem~\ref{thm:gpd_operad}. Appendix~\ref{ap:routine_verifications} contains somewhat tedious computational verifications, put there so as not to interrupt the flow of the exposition. Finally, Appendix~\ref{ap:structure_map_Kan} is somewhat auxiliary to the main narrative of the paper; it establishes a sufficient condition for the structure map to be a Kan fibration, stated as Theorem~\ref{thm:structure_map_Kan}.


\subsection*{Acknowledgments} The author would like to thank Boris Shoikhet, Viktor Lavrukhin, and Vasily Ionin for helpful discussions. Additionally, the author wishes to express sincere gratitude to Matthew Magin for careful reading of a draft of this paper, and to Andrey Ryabichev
for drawing the beautiful pictures for this paper.
% ...


\section{Functorial classifying space}
\label{sec:functorial_classifying_space}
\subsection{Preliminaries}

For every crossed simplicial group $G_*$ we denote its associated category by $\Delta G$, as in~~\cite{f0e21457-c8c3-31f9-bd48-762cf5f699ae}.

\begin{definition}
    For a crossed simplicial group $G_*$ and a category $\cc$ a \textit{$G_*$-object} in $\cc$ is a functor $(\Delta G)^\mathrm{op} \to \cc$. 
\end{definition}
The following interpretation of this notion is useful for us.
\begin{lemma}[Lemma~4.2 of~{\cite{f0e21457-c8c3-31f9-bd48-762cf5f699ae}}]
\label{action_lemma_fied_lod}
The data of a $G_*$-object in $\cc$ is equivalent to that of a simplicial object $X$ in $\cc$ together with given left actions $G_n \curvearrowright X_n$ for each $n$ satisfying the following identities{\rm:}
\begin{enumerate}        
    \item \textbf{Face identities}{\rm :} $d_i(gx) = d_i(g)d_{g^{-1}(i)}(x)$;
    \item \textbf{Degeneracy identities}{\rm :} $s_i(gx) = s_i(g)s_{g^{-1}(i)}(x)$.
\end{enumerate}
\end{lemma}

We remind the reader the structure theorem of crossed-simplicial groups.
\begin{theorem}[Theorem~3.6 of~{\cite{f0e21457-c8c3-31f9-bd48-762cf5f699ae}}]\label{structure_thm_fied_lod}
For any crossed simplicial group $G_*$ there is a short exact sequence of crossed simplicial groups
$$1 \to P_* \to G_* \xrightarrow{\pi} N_* \to 1,$$
where $P_*$ is a simplicial group and $N_*$ is one of the following {\rm 7} crossed simplicial groups{\rm:}
$$\{1\}, \; C_*, \; \{\ZZ/2\}, \; S_*, \; D_*, \; {\ZZ/2\times S_*}, \; H_*.$$
\end{theorem}
\begin{definition}
For a group $G$ and a normal subgroup $H \subseteq G$, the \textit{action groupoid} $G\sslash H$ is a groupoid with
\begin{enumerate}
\item objects being the elements of the group $G/H$;
\item the hom-sets $\mathrm{Hom}_{G \sslash H}(a, b)$ consisting of the elements $f \in G$ such that $af^{-1} = b$, where $f$ on the left hand side of the equality is identified with its image in $G/H$. 
\end{enumerate}
\end{definition}
\subsection{Main results of the section}
Fix a crossed simplicial group $G_*$ with the structural decomposition as in Theorem \ref{structure_thm_fied_lod}. Set
\begin{equation}
\label{eq:gamma_def}
\Gamma_n = \Gamma_n(G) = G_n \sslash P_n
\end{equation}
to be the action groupoid. We formulate our first theorem.

\begin{theorem}\label{thm:groupoid_construction}
Let $G_*$ be a crossed simplicial group with the decomposition as in Theorem~\ref{structure_thm_fied_lod}.
Then, the sequence of groupoids~\eqref{eq:gamma_def} forms a functor $\Gamma_* \colon (\Delta N)^\mathrm{op} \to \Gpd$. Moreover, for each $n$ the action $N_n \curvearrowright \Gamma_n$ is free and the orbit space $N_n \backslash \B\Gamma_n$ has the homotopy type of $K(G_n, 1)$.
\end{theorem}
The theorem above would imply the following.
\begin{corollary}\label{thm_corollary}
The sequence of spaces $\{\B\Gamma_n\}_{n \ge 0}$ forms a $N_*$-space such that 
\begin{enumerate}
\item The orbit space $N_n \backslash \B\Gamma_n$ has the homotopy type of $K(G_n, 1)$ for each \(n\);
\item The sequence universal covering spaces $\{\B\Gamma_n\langle 1\rangle\}_{n \ge 0}$ form a $\Delta G_*$-space such that the action \(G_n \curvearrowright B\Gamma_n\langle 1 \rangle\) is free for each \(n\).
That is, $\B\Gamma_n\langle 1 \rangle$ is a $EG_n$ space.
\end{enumerate}
\end{corollary}
We prove these two results below.


\subsection{Proof of Theorem~\ref{thm:groupoid_construction}}

In the action groupoid~$\Gamma_n$ we denote a morphism $\sigma \to \tau$ given by $f \in G_n$ with the symbol
$$[\sigma, f] \in \Gamma_n(\sigma, \sigma f^{-1}).$$
Then, the composition can be described by the formula
$$[\sigma f^{-1}, g] \cdot [\sigma, f] = [\sigma, gf].$$
Next, we verify that these groupoids form a simplicial groupoid. 

\begin{construction}
For each $i = 0, \ldots, n$ define $s_i \colon \Gamma_n \to \Gamma_{n+1}$ by the formula
\begin{equation}\label{degeneracy_groupoid}
s_i[\sigma, f] = [s_i(\sigma), s_{\sigma^{-1}(i)}(f^{-1})^{-1}].
\end{equation}
This defines a correct map~$\mathrm{Hom}_{\Gamma_n}(\sigma, \tau) \to \mathrm{Hom}_{\Gamma_{n+1}}(s_i(\sigma), s_i(\tau))$, since for every \([\sigma, f] \colon \sigma \to \tau\) one has
\begin{align*}
\mathrm{dom}_{\Gamma_{n+1}}(s_i([\sigma, f])) &= \mathrm{dom}_{\Gamma_{n+1}}([s_i(\sigma), s_{\sigma^{-1}(i)}(f^{-1})^{-1}]) = s_i(\sigma), \\
\mathrm{cod}_{\Gamma_{n+1}}(s_i([\sigma, f])) &= \mathrm{cod}_{\Gamma_{n+1}}([s_i(\sigma), s_{\sigma^{-1}(i)}(f^{-1})^{-1}]) = s_i(\sigma) \cdot s_{\sigma^{-1}(i)}(f^{-1}) = s_i(\sigma f^{-1}) = s_i(\tau). \\
\end{align*}
(This double inversion in~\eqref{degeneracy_groupoid} is necessary because the degeneration maps $s_* \colon G_n \to G_{n-1}$ are not group homomorphisms.)

Similarly, define the face maps $d_i \colon \Gamma_n \to \Gamma_{n-1}$ by the formula 
\begin{equation}\label{face_groupoid}
d_i[\sigma, f] = [d_i(\sigma), d_{\sigma^{-1}(i)}(f^{-1})^{-1}].
\end{equation}
\end{construction}

\begin{lemma}\label{simplicial_groupoid}
The maps defined in {\rm(}\ref{degeneracy_groupoid}{\rm)} and {\rm(}\ref{face_groupoid}{\rm)} give $\{\Gamma_n\}_{n \geq 0}$ the structure of a simplicial groupoid. 
\end{lemma}
\begin{proof}
First, we claim that $d_i$ and $s_i$ are functors of groupoids.
We show this for degeneracies only, and the case of faces is completely analogous.
Note that
$$s_i([\sigma f^{-1}, g] \cdot [\sigma, f]) = s_i([\sigma, gf]) = [s_i(\sigma), s_{\sigma^{-1}(i)}((gf)^{-1})^{-1}],$$
while
$$s_i([\sigma f^{-1}, g]) \cdot s_i([\sigma, f]) = [s_i(\sigma f^{-1}), s_{(\sigma f^{-1})^{-1}(i)}(g^{-1})^{-1}] \cdot [s_i (\sigma), s_{\sigma^{-1}(i)}(f^{-1})^{-1}].$$
Therefore, functoriality requires the identity
$$s_{\sigma^{-1}(i)}((gf)^{-1})^{-1} = s_{(\sigma f^{-1})^{-1}(i)}(g^{-1})^{-1} \cdot s_{\sigma^{-1}(i)}(f^{-1})^{-1},$$
which is equivalent to 
$$s_{\sigma^{-1}(i)}(f^{-1}g^{-1}) = s_{\sigma^{-1}(i)}(f^{-1}) \cdot s_{f \cdot \sigma ^{-1}(i)}(g^{-1}). $$
This latter follows from the crossed group axioms. The claim is proved. We verify the simplicial idenitities in the Appendix \ref{ap:routine_verifications}. \end{proof}

Thus, the sequence $\{\Gamma_n\}_{n \ge 0}$ forms a simplicial groupoid $\Gamma_* \colon \Delta^\mathrm{op} \to \Gpd$.
Moreover, the group $N_n$ acts on $\Gamma_n$ on the left via the formula
$$\tau \cdot [\sigma, f] = [\tau \sigma, f].$$

We claim that these actions define a functor $\Gamma \colon (\Delta N)^\mathrm{op} \to \Gpd$. We will verify the required identities from Lemma \ref{action_lemma_fied_lod}.
\begin{enumerate}
\item Face identities: for $\sigma \in N_n$ and $x \in \Gamma_n$ one has
$$d_i(\sigma x) = d_i(\sigma)d_{\sigma^{-1}(i)}(x).$$
Indeed,
\begin{align*}
d_i([\tau \sigma, f]) &= [d_i(\tau) \cdot d_{\tau^{-1}(i)}(\sigma), d_{\sigma^{-1}\tau^{-1}(i)}(f^{-1})^{-1}] \\ &= d_i(\tau) \cdot [ d_{\tau^{-1}(i)}(\sigma), d_{\sigma^{-1}\tau^{-1}(i)}(f^{-1})^{-1}] \\ &= d_i(\tau) \cdot d_{\tau^{-1}(i)}([\sigma, f]).
\end{align*}
    
\item Degeneracy identities: for $\sigma \in N_n$ and $x \in \Gamma_n$ one has $$s_i(\sigma x) = s_i(\sigma ) \cdot s_{\sigma^{-1}(i)}(x),$$
which can be checked analogously. 
\end{enumerate}

The homotopy type of $\B\Gamma_n$ is that of $K(P_n, 1)$, since $\Gamma_n$ is a connected groupoid with fundamental group $P_n$ and groupoids are $1$-truncated.

\begin{lemma}\label{quotient_G}
The quotient space $N_n\backslash \B\Gamma_n$ has the homotopy type of $K(G_n, 1)$. 
\end{lemma}
\begin{proof}
We work simplicially. We show that  the simplicial set $N_n \backslash N\Gamma_n$ is isomorphic to the nerve $N(G_n^\mathrm{op})$.
It would imply the statement of lemma, since
\[|NG_n^\mathrm{op}| \cong |NG_n| \simeq K(G_n, 1).\]
    
An $m$-simplex in the nerve $N\Gamma_n$ is determined by a starting object $\sigma \in N$ and a sequence $f_1, \ldots, f_m \in G_n$ corresponding to the composite morphism
        $$[\sigma f_1^{-1} f_2^{-1}\ldots\cdot f_{m-1}^{-1}, f_m]\cdot\ldots\cdot [\sigma f_1^{-1}, f_2] \cdot [\sigma, f_1] $$
The $N_n$-action quotients out the objects $\sigma \in N_n$, leaving precisely the simplicial set $NG_n$. More explicitly, the map $$\{[\sigma f_1^{-1} f_2^{-1}\ldots\cdot f_{m-1}^{-1}, f_m]\cdot\ldots\cdot [\sigma f_1^{-1}, f_2] \cdot [\sigma, f_1]\}_{\sigma \in N_n} \mapsto (f_m, \ldots, f_1)$$
gives a bijection $(N_n\backslash N\Gamma_n)_m \to N(G_n^\mathrm{op})_m$. A routine check shows that it is a simplicial mapping.
\end{proof}
This completes the proof of the theorem.

\subsection{Proof of Corollary~\ref{thm_corollary}}
The first claim of Corollary~\ref{thm_corollary} trivially follows from Theorem~\ref{thm:groupoid_construction}.

To address the second claim, we first clarify our use of the term \emph{covering space}. There are two possible approaches. The first follows Fiedorowicz’s construction of the universal cover of the configuration space, which involves selecting an appropriate subspace of the universal cover to act as a ``generalized basepoint''. The second approach, which we adopt, proceeds within the category of groupoids. 

For a connected groupoid $\Gamma$, a standard functorial construction of its covering is given by the slice groupoid $x_0 / \Gamma$ for a fixed object $x_0$. This construction is monoidal functorial in the category of pointed groupoids. However, our functor $\Delta N \to \Gpd$ cannot be regarded as taking values in pointed groupoids, since the $N_*$--action permutes the objects transitively. 

To circumvent this, we consider the composition
\[
\Gamma\langle 1 \rangle \colon \Delta^{\mathrm{op}} \to \Gpd_* \xrightarrow{\text{slice}} \Gpd_*.
\]
We claim that $\Gamma\langle 1 \rangle$ extends naturally to a functor 
\[
(\Delta G)^{\mathrm{op}} \to \Gpd.
\]
This extension is straightforward to verify, and we omit the details.


\section{Operadic crossed simplicial groups}\label{sec:operadic_crossed_simplicial_groups}

The aim of this section is to introduce several enhancements of the notion of a crossed simplicial group, and to demonstrate some examples.

Let $S_*$ be the crossed simplicial symmetric group.  Recall that $S_n = \Sigma_{n+1} \overset{\mathrm{def}}{=} \mathrm{Bij}(\{0, \ldots, n\})$, the group of bijections of the set~$[n]$. First, we wish to discuss some additional structure that is present on the crossed simplicial group $S_*$. 


\begin{enumerate}
  \item There is the left contracting simplicial homotopy $s_L$, which adds a fixed point to the left:
  $$s_L \colon \sigma \mapsto (1, \sigma(0) + 1, \sigma(1)+1, \ldots, \sigma(n)+1).$$
  \item There is the right conctracting simplicial homotopy $s_R$, which adds a fixed point to the right:
  $$s_R \colon \sigma \mapsto (\sigma(0), \sigma(1), \ldots, \sigma(n), n+1).$$
  \item These two contracting homotopies define the block sum of permutations: for $\sigma \in S_n, \tau \in S_m$ the block sum is the permutation $\sigma \boxplus \tau \in S_{n + m + 1}$ given by the sequence
  $$\sigma(0), \ldots, \sigma(n), \tau(0) + n + 1, \ldots, \tau(m) + n + 1.$$
  This is a group homomorphism $S_n \times S_m \to S_{n + m + 1}.$
  \item Lastly, there is the standard (see ~\cite[Proposition~1.1.9.]{Fresse-book}) $\Set$-operad structure on the sequence $\{S_n\}$. The definition itself is completely analogous to the one we will give in Theorem~\ref{thm:set_operad}.
\end{enumerate}
Now we introduce the first enhancement.

\begin{definition}
\label{def:monoidal_csg}
\begin{enumerate}
    \item A left contractible symmetric simplicial group is a pair $(G, s_L)$, where $G$ is a symmetric simplicial group and $s_L = \{s_{-1}\}_{n \geq 0}$ is a left contraction homotopy\footnote{For more information on simplicial contractions, see \cite{Barr2019}}, such that it is a group homomorphism, making the following diagram commutative 
\[\begin{tikzcd}
  {G_n} && {G_{n+1}} \\
  {S_n} && {S_{n+1}.}
  \arrow["{s_L}", from=1-1, to=1-3]
  \arrow["\pi"', from=1-1, to=2-1]
  \arrow["\pi", from=1-3, to=2-3]
  \arrow["{s_L}"', from=2-1, to=2-3]
\end{tikzcd}\]
\item Right contractible symmetric simplicial group is a pair $(G, s_R)$, where $G$ is a symmetric simplicial group and $s_R = \{s_{n+1}\}_{n \geq 0}$ is a right contraction homotopy, such that it is a group homomorphism, making an analogous diagram commutative. 
\item Ambi contractible symmetric simplicial group is a triple $(G, s_L, s_R)$ such that $(G, s_L)$ is a left contractible symmetric simplicial group and $(G, s_R)$ is a right contractible symmetric simplicial group. 
\end{enumerate}
\end{definition}
In what follows, by a slight abuse of notation, we call a symmetric group $G_*$ ambi contractible, tacitly fixing the maps $s_L$ and $s_R$ and omitting them from the notation.

The second enhancement of the notion of a crossed simplicial groups comes from the observation that the block sum of permutations may be expressed in terms of the operators $s_L$ and $s_R$. Explicitly, for $\sigma \in S_n$ and $\tau \in S_m$, one has $\sigma \boxplus \tau = (s_R)^{\circ m + 1}(\sigma) \cdot (s_L)^{\circ n + 1}(\tau)$. 

\begin{definition}\label{def:monoidal_p_csg}
A symmetric ambi contractible crossed simplicial group $G_*$ is said to be \textit{monoidal} if the map
\[
\boxplus \colon G_n \times G_m \longrightarrow G_{n+m+1}, \qquad
(\alpha,\beta) \longmapsto (s_R)^{\circ (m+1)}(\alpha)\cdot (s_L)^{\circ (n+1)}(\beta),
\]
is a group homomorphism. Equivalently, $G_*$ is \textit{monoidal} if and only if the elements $(s_R)^{\circ (m+1)}(\alpha)$ and $(s_L)^{\circ (n+1)}(\beta)$ commute for all $\alpha \in G_n$ and $\beta \in G_m$.

We further introduce the notation $1_n$ for the identity element of $G_n$; thus 
\[
\alpha \boxplus 1_m = (s_R)^{\circ m}(\alpha)
\qquad\text{and}\qquad
1_n \boxplus \beta  = (s_L)^{\circ n}(\beta).
\]
\end{definition}

The following lemma is just unpacking the definition. 
\begin{lemma}
    Let $G_n$ be a monoidal crossed simplicial group. Then 
    \begin{enumerate}
        \item The disconnected groupoid $\bigsqcup_{n \geq 0} G_n$ is a monoidal category under $\boxplus$;
        \item The space $\bigsqcup_{n \geq 0} \B G_n$ is a strict topological monoid under $\B\,\boxplus$. 
    \end{enumerate}
\end{lemma}

Before proceeding to the examples, we introduce this paper's last enhancement of the notion of a crossed simplicial group.

\begin{definition}
\label{def:operadic_csg}
A monoidal crossed simplicial group $G_*$ is called \textit{operadic} if the following identity holds for all $\alpha \in G_n, \beta \in G_m$ and for all $0 \leq i \leq n$:
    $$(1_i \boxplus \beta \boxplus 1_{n - i})\cdot s_i^m(\alpha) = s_{i}^m(\alpha)\cdot (1_{\alpha^{-1}(i)} \boxplus \beta \boxplus 1_{n - \alpha^{-1}(i)}).$$
\end{definition}

\subsection{Examples}\label{sec:examples}

The primary class of examples considered in this paper is the family of braid-like groups in the broad sense.  The most prominent and motivating example is the Artin braid group \(B_n\); for an introduction in braid theory we refer the reader to ~\cite{gonzalezmeneses2010basicresultsbraidgroups}. In this paper, we adopt the following convention: the group $Br_n$ denotes the Artin braid group on $n$ strands, and we set \(B_n := Br_{n+1}\).

We recall the conventions and describe how the structural maps introduced above are realised for the braid groups.  Throughout this subsection we regard an \(n\)-strand braid as indexed by \([n-1]\),
so that the top and bottom endpoints of the braid are labelled by the same ordered
set.  
Multiplication in \(B_n\) is given by vertical concatenation (see Fig.~\ref{fig:basics}). The natural left action
\[
B_n\curvearrowright [n]
\]
is given by this permutation action: \(\beta\in B_{n}\) sends the label \(i\in[n]\) to the index of the strand that ends in the \(i\)-th bottom position. This defines the canonical projection (which gives the structural projection)
\[
\pi_n\colon B_n\twoheadrightarrow S_n,
\]
which forgets the crossing data and records only the permutation of endpoints induced by a braid. 

\begin{figure}[H]
  \centering
  \includegraphics[width=0.8\textwidth]{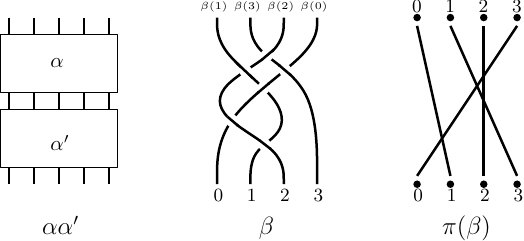}
  \caption{Multiplication by vertical concatenation and the action on \([n]\).}
  \label{fig:basics}
\end{figure}

The simplicial maps \(s_i\colon B_{n}\to B_{n+1}\) and \(d_i\colon B_{n+1}\to B_n\)
are given by duplicating and deleting the \(i\)-th strand, respectively
(see Fig.~\ref{fig:degeneracy_face}).  With these conventions, the family
\(\{B_{n}\}_{n\ge 0}\) forms a crossed simplicial group (see~\cite{f0e21457-c8c3-31f9-bd48-762cf5f699ae}).
\begin{figure}[H]
  \centering
  \includegraphics[width=0.8\textwidth]{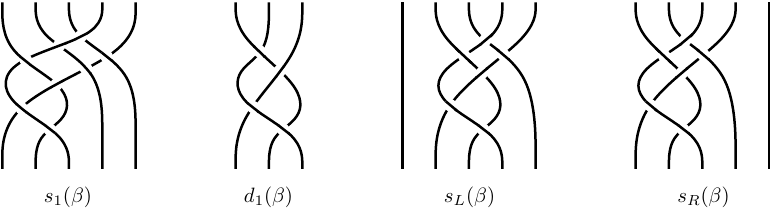}
  \caption{The degeneracy \(s_i\) (cabling the \(i\)-th strand), the face \(d_i\) (deleting the \(i\)-th strand), and two contracting homotopies $s_L, s_R$ applied to the braid $\beta$ from Fig.~\ref{fig:basics}.}
  \label{fig:degeneracy_face}
\end{figure}

This crossed simplicial group is ambi contractible: there exist two distinguished degeneracy-type maps
\[
s_L,s_R\colon B_{n}\longrightarrow B_{n+1}
\]
which add a new trivial strand at the left-most (resp.\ right-most) position (see Fig.~\ref{fig:degeneracy_face}).  


It is easy to see that the ``monoidal'' axiom (Definition~\ref{def:monoidal_p_csg}) is satisfied and the product of two braids is given by the juxtaposition operation (see Fig.~\ref{fig:boxplus}).

\begin{figure}[H]
  \centering
  \includegraphics[width=0.8\textwidth]{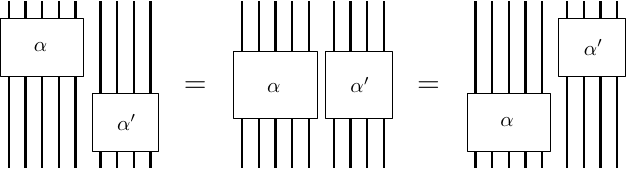}
  \caption{Horizontal juxtaposition \(\alpha\boxplus\alpha'\) of two braids.}
  \label{fig:boxplus}
\end{figure}

The fact that, under these structural maps, the crossed simplicial braid group is
operadic follows from a straightforward identity in the braid group.  This can be
seen directly from the diagrammatic representation (Fig.~\ref{fig:identbraid}),
where the compatibility of the structural maps is visually evident.  A rigorous verification may be obtained by appealing to the standard braid group
fact that \(\sigma_i^{\,b} = \sigma_j\) if and only if \(b^{-1}(i) = j\) and a
ribbon can be inserted between the \(i\)-th and \((i+1)\)-st strands of \(b\); see, for example,~~\cite{dehornoy2013foundationsgarsidetheory}.


\begin{figure}[H]
\centering
\begin{minipage}{.5\textwidth}
 \vspace{0pt}
  \centering
  \includegraphics[width=.9\linewidth]{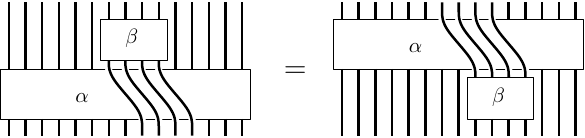}
\end{minipage}%
\begin{minipage}{.5\textwidth}
\vspace{0pt}
  \centering
  \includegraphics[width=.9\linewidth]{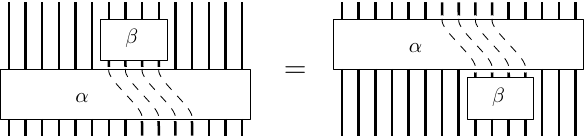}
\end{minipage}
\caption{Relations in \(B_n\) illustrating the operadicity axiom.}
\label{fig:identbraid}
\end{figure}


\subsubsection{Other braid-like groups}

Another class of examples is provided by the \emph{homotopy braid groups} \(hBr_n\) 
(see, e.g.,~~\cite{bardakov2021homotopybraids} for a formal treatment).  The structural 
decomposition is
\[
1 \longrightarrow hP_{n+1} \longrightarrow hBr_{n+1} \longrightarrow S_n \longrightarrow 1,
\]
where \(hP_n\) denotes the pure homotopy braid group.  The usual doubling and deletion 
of strands endows \(hB_n\) with a crossed simplicial group structure.  Since \(\{hBr_{n+1}\}_{n \geq 0}\) 
is a quotient of the standard braid crossed simplicial group, all the structural axioms 
(ambi contractible, monoidal, operadic) are inherited.

A further example is the \emph{welded braid group} \(WBr_n\) (see, e.g., ~\cite{bellingeri2020noterepresentationsweldedbraid}), with structural decomposition
\[
1 \longrightarrow WP_{n+1} \longrightarrow WBr_{n+1} \longrightarrow S_n \longrightarrow 1,
\]
where \(WP_n\) is the pure welded braid group.  Ambi-contractibility and monoidicity are immediate, and operadicity follows by an argument analogous to that for the Artin braid groups, using a similar criterion for conjugation of generators (see ~\cite{Bellingeri31122004}).


\section{Constructing operads}\label{sec:constructing_operads}
 Our main goal is to define operads using crossed simplicial groups. The simplest case is that of an operad in sets. To make the subsequent constructions self-contained, we briefly recall the definition of an operad and introduce a shifted variant that will be more convenient for our purposes.

\subsubsection{Operads}
\begin{definition}[see ~\cite{Loday-algebraic-op}]\label{definition_operad}
    Let $\mm$ be a monoidal category. A (non-symmetric) \textit{operad} in $\mm$ is a sequence of objects $\OO(0), \OO(1), \OO(2), \ldots$ with the following additional data: for each $n \geq 1, m \geq 0$ there are $n$ operations $\circ_1, \ldots, \circ_n$ 
        $$\OO(n) \times \OO(m) \to \OO(n + m - 1).$$
    And they are subject to the following relations and constraints:
    
    \begin{enumerate}
        \item There is only one element $* \in \OO(0)$;
        \item there is an element $\id \in \OO(1)$, such that     
        $$\id \circ_1 \nu = \nu, \mu \circ_i \id = \mu;$$
        \item for $\lambda \in \OO(l), \mu \in \OO(m), \nu \in \OO(n)$         one has $$(\lambda \circ_i \mu) \circ_{i + j - 1} \nu = \lambda \circ_i(\mu \circ_j \nu)$$ for $1 \leq i \leq l, \; 1 \leq j \leq m$;
        \item for $\lambda \in \OO(l), \mu \in \OO(m), \nu \in \OO(n)$         one has $$(\lambda \circ_i \mu) \circ_{k - 1 + m} \nu  = (\lambda \circ_k \nu ) \circ_i \mu$$
        for $1 \leq i < k \leq l$.
    \end{enumerate}
\end{definition}

In what comes next, using a crossed simplicial group $G_*$ we will build an operad with $\OO(n) = G_{n-1}$ (or a groupoid built using $G_{n-1}$). To maintain smooth notation we propose an (equivalent) notion of a shifted operad.

\begin{definition}
    Let $\mm$ be a monoidal category. A (non-symmetric) \textit{shifted operad} in $\mm$ is a semi-simplicial object $\uu(*) \in \mm^{\Deltainj^\mathrm{op}}$ together with $n+1$ morphisms $\circ_0, \ldots, \circ_n$ acting $$\circ_i \colon  \uu(n) \times \uu(m) \to \uu(n+m) \text{ for } i = 0, \ldots, n.$$ Subject to the following relations and constraints: 
    \begin{enumerate}
        \item There is an element $\id \in \uu(0)$ such that 
         $$\id \circ_0 \nu = \nu,\quad  \mu \circ_i \id = \mu;$$
        \item for $\lambda \in \uu(l), \mu \in \uu(m), \nu \in \uu(n)$ one has 
        $$(\lambda \circ_i \mu) \circ_{i + j} \nu = \lambda \circ_i(\mu \circ_j \nu)$$ for $0 \leq i \leq l, \; 0 \leq j \leq m$;

        \item for $\lambda \in \uu(l), \mu \in \uu(m), \nu \in \uu(n)$ one has $$(\lambda \circ_i \mu) \circ_{k + m} \nu  = (\lambda \circ_k \nu ) \circ_i \mu$$
        for $0 \leq i < k \leq l$;

        \item for $\lambda \in \uu(l), \mu \in \uu(m)$ one has 
  $$d_{i + j}(\lambda \circ_i \mu) = \lambda \circ_i d_j(\mu )$$
  for $0 \leq i \leq l, 0 \leq j \leq m$;
  \item for $\lambda \in \uu(l), \mu \in \uu(m), \nu \in \uu(n)$ and $0 \leq i < k \leq l$ one has 
  $$d_i(\lambda \circ_k \nu) = d_i(\lambda) \circ_{k} \nu,$$
  $$d_{k + m}(\lambda \circ_i \mu)  = d_k(\lambda) \circ_i \mu.$$
\end{enumerate}   
\end{definition}
The next lemma is straightforward, so we omit its proof.
\begin{lemma}
    The transformation $$\uu(n) \mapsto \OO(n + 1), \quad \circ_i \mapsto \circ_{i + 1}, \quad  d_i \mapsto (-) \circ_{i + 1} *$$ provides an isomorphism between the categories of operads and shifted operads.
\end{lemma}

 
 Fix a monoidal (ambi contractible, symmetric) crossed simplicial group $G_*$. 

\begin{theorem}\label{thm:set_operad}
    Define $n + 1\,$ $\Set$ functions $G_n \times G_m \to G_{n + m}$ by the formula
    $$\circ_i \colon  (\alpha, \beta) \mapsto (1_i \boxplus \beta \boxplus 1_{n - i}) \cdot s_i^m(\alpha).$$

    Under these operations and $d_i$'th from the simplicial structure, the sets $G_*$ form a {\rm(}non-symmetric\!~{\rm)} operad in $\Set$.
\end{theorem}
The proof of this theorem is fully computational and is given in the appendix. 

Following ~\cite{Fied-preprint}, we define $G_*$-like operads in a monoidal category $\mm$. While the definition below is formulated for shifted operads, it admits a straightforward translation to the language of ordinary operads. In particular, for the crossed simplicial group $S_*$, our definition recovers the classical notion of a symmetric operad, whereas for $B_*$ it coincides with Fiedorowicz's notion of a braided operad.


\begin{definition}
    A $G_*$-like shifted operad in $\mm$ is a non-symmetric shifted operad $\uu(n)$ such that each object $\uu(n)$ is equipped with a right action of $G_n$ subject to the following equivariance conditions. For each $\mu \in \uu(m), \nu \in \uu(n)$, $\beta \in G_n$ and $i \leq m$: 
    \begin{enumerate}
        \item $\mu \circ_i \nu^\beta = (\mu \circ_i \nu)^{\beta'}$, where $\beta' = 1_{i} \boxplus \beta \boxplus 1_{m - i}$;
        \item $\mu^\beta \circ_i \nu = (\mu \circ_{\sigma(i)}\nu)^{\beta''}$, where $\beta'' = d_i^n \beta$.
    \end{enumerate}
\end{definition}

As in the braided case (see~\cite{Fied-preprint}), we define a $G_\infty$-operad to be a $G_*$-like operad for which the spaces $\OO(n)$ are contractible and the actions of $G_n$ are free. 

To proceed with Fiedorowicz's arguments, it is necessary to provide a specific example of a $G_*$-like operad with $\OO(n) = EG_n$. Remarkably, the groupoids introduced in the first paragraph of this paper serve this purpose.

Before we proceed to the construction of the operad in groupoids, we observe that the operadicity axiom (Definition~\ref{def:operadic_csg}) imposed on a symmetric ambi contractible crossed simplicial group $G_*$ implies the following relation in the $\Set$-operad for all admissible inputs. 
\begin{equation}\label{operadic-mult}
  (\alpha \circ_i \beta) \cdot (\alpha' \circ_{\alpha^{-1}(i)} \beta') = \alpha\alpha' \circ_i \beta\beta'.  
\end{equation}

\begin{proof}[Proof of the identity \ref{operadic-mult}]
    Consider $\alpha, \alpha' \in G_n$, $\beta, \beta' \in G_m$. Unfold the left-hand-side and use the operadicity axiom on the second and third factors: 
    $$(1_i \boxplus \beta \boxplus 1_{n-i}) \cdot s_i^m(\alpha) \cdot (1_{\alpha^{-1}(i)} \boxplus \beta' \boxplus 1_{n - \alpha^{-1}(i)})\cdot s_{\alpha^{-1}(i)}^m(\alpha') = $$
    $$ = (1_i \boxplus \beta \boxplus 1_{n-i}) \cdot (1_i \boxplus \beta' \boxplus 1_{n-i}) \cdot s_i^m(\alpha)s_{\alpha^{-1}(i)}^m(\alpha').$$
\end{proof}

\begin{theorem}\label{thm:gpd_operad}
    Let $G_*$ be an operadic crossed simplicial group and let $\Gamma_*$ be the sequence of its action groupoids. Define $n + 1$ functions $\circ_i \colon \Gamma_n \times \Gamma_m \to \Gamma_{n + m}$ as 
    $$[\sigma, f] \circ_i [\tau, g] = [\sigma \circ_i \tau, (f^{-1} \circ_{\sigma^{-1}(i)} g^{-1})^{-1}],$$
    where $\circ_{\sigma^{-1}(i)}$ is the operation on $G_*$ given by Theorem~\ref{thm:set_operad}.
    Then $\circ_i$ are functors of groupoids. Moreover, the sequence of groupoids $\{\Gamma_n\}$ with $\circ_i$'s and simplicial $d_j$'s form a shifted operad in $\Gpd$.
\end{theorem}

\begin{proof}
    First, one requires that $\circ_i$'s are correctly defined functors of groupoids. This follows exactly from the identity \ref{operadic-mult}. The axioms of a shifted operad follow from Theorem \ref{thm:set_operad}.
\end{proof}
\begin{remark}\label{rem:E_ooE_2}
It is worth noting, that for $G_* = B_*$ and $G_* = S_*$, the constructions appearing in the above theorem have already been studied in the literature. In the case $G = S_*$, the spaces $\B\Gamma(S_*) = ES_*$ provide the standard model for the objects of the $E_\infty$ operad. For $G = B_*$, the resulting operad specializes to the model for the $E_2$ operad described in \cite[Chapter~5]{Fresse-book}.
\end{remark}

\subsection{Applying these results}

Let $G_*$ be a monoidal crossed simplicial group. Denote by $\OO(n)$ the $G_\infty$ operad built above, let $T_G$ be its monad $T_G \colon \Top \to \Top$. The first result is a tautology.
\begin{theorem}
    Consider the group completion $T_G(*)^{grp}$. Then there is a weak equivalence 
    $$T_G(*)^{grp} \xrightarrow{\simeq} \Omega \B\left(\bigsqcup_{n \geq 0} \B G_n\right).$$
\end{theorem}
\begin{proof}
After checking that necessary structures align we are left with a simple rewriting
    $$T_G(*)^{grp} = \Omega \B(T_G(*)) = \Omega \B\left(\bigsqcup_{n \geq 0} EG_n \times_{G_n} *\right) = \Omega \B\left(\bigsqcup_{n \geq 0} \B G_n\right). \qedhere$$
\end{proof}


The following statements can be justified by arguments analogous to those in~\cite{Fied-preprint}. Define the $G$-bar construction for a topological monoid $M$ as the composition functor 
$$\B^{G_*}M \colon \Delta G \to \Delta S \to \Top, \quad [n] \mapsto M^{n + 1}.$$
We thus obtain an analogue of Fiedorowicz’s theorem in this more general setting.

\begin{theorem}[cf.~~\cite{Fied-preprint}]
There is a weak equivalence 
    $$\hocolim_{\Delta G} \B^{G_*}M \simeq \B(T_G, C_1, M).$$
\end{theorem}

And the according corollary.
\begin{corollary}
    For any pointed space $X$ there is the weak equivalence
    $$\hocolim_{\Delta G} \B^{G_*}(JX) \simeq T_GX, $$
    where $JX$ denotes James construction.
\end{corollary}

    

\appendix
\section{Routine verifications}\label{ap:routine_verifications}

\begin{lemma}\label{symm_identities}
The following identities hold with regards to the standard simplicial set and operad structures on $S_*$\footnote{Recall that \(S_n\) is the symmetric group \(\mathrm{Bij}([n])\) of rank \(n+1\)}:
\begin{enumerate}
\item For $i < j$ one has 
$$d_i(\sigma)^{-1}(j-1) = \begin{cases} \sigma^{-1}(j) -1 \text{, if } \sigma^{-1}(i) < \sigma^{-1}(j), \\
\sigma^{-1}(j)\text{, if } \sigma^{-1}(j) < \sigma^{-1}(i).\end{cases}$$
\item For $i < j$ one has
$$d_j(\sigma)^{-1}(i) = \begin{cases}
\sigma^{-1}(i) - 1 \text{, if } \sigma^{-1}(j) < \sigma^{-1}(i), \\ 
\sigma^{-1}(i) \text{, if } \sigma^{-1}(i) < \sigma^{-1}(j).\end{cases}$$
\item For $i < j$ one has
$$s_j(\sigma)^{-1}(i) = \begin{cases} \sigma^{-1}(i)+1\text{, if } \sigma^{-1}(j) < \sigma^{-1}(i), \\ \sigma^{-1}(i)\text{, if } \sigma^{-1}(i) < \sigma^{-1}(j).\end{cases}$$    
\item For $\sigma \in S_n, \tau \in S_m$ and for $i \in [n], j \in [m]$ one has
$$(\sigma \circ_{i} \tau)^{-1}(i + j) = \sigma^{-1}(i) + \tau^{-1}(j).$$ 
\item For $i < j$ one has
$$s_i(\sigma)^{-1}(j + 1) = \begin{cases}\sigma^{-1}(j)\text{, if } \sigma^{-1}(j) < \sigma^{-1}(i) \\ 
\sigma^{-1}(j)+1\text{, if } \sigma^{-1}(i) < \sigma^{-1}(j)\end{cases}$$
\end{enumerate}
\end{lemma}
\begin{proof}
Straightforward verification.
\end{proof}

\subsection{Verification of simplicial identities in Lemma~\ref{simplicial_groupoid}}
\label{sec:appendix_simplicial_identities}
We prove that simplicial identities hold in $\Gamma_*$. 
\begin{enumerate}
\item The case $d_id_j = d_{j-1}d_i$ for $i < j.$
Fix $[\sigma, f] \in \Gamma_n$. Then,
$$d_id_j [\sigma, f] = d_i[d_j(\sigma), d_{\sigma^{-1}(j)}(f^{-1})^{-1}] = 
[d_id_j(\sigma),
d_{d_j(\sigma)^{-1}(i)}(d_{\sigma^{-1}(j)}(f^{-1}))^{-1}].
$$
The first component clearly satisfies the simplicial identity. For the second, we must demonstrate the equality
\begin{equation}
\label{eq:didj_identity}
d_{d_j(\sigma)^{-1}(i)}(d_{\sigma^{-1}(j)}(f)) = d_{d_i(\sigma)^{-1}(j-1)}(d_{\sigma^{-1}(i)}(f))
\end{equation}
for arbitrary $f$.
This can be verified by considering the two cases $\sigma^{-1}(j) < \sigma^{-1}(i)$ and $\sigma^{-1}(i) < \sigma^{-1}(j)$. 
In the first case $\sigma^{-1}(j) < \sigma^{-1}(i)$, the desired identity~\eqref{eq:didj_identity} transforms to the equation
$$d_{\sigma^{-1}(i) - 1}(d_{\sigma^{-1}(j)}(f)) = d_{\sigma^{-1}(j)}(d_{\sigma^{-1}(i)}(f))$$
using the formulas from Lemma~\ref{symm_identities}.
And this equation is an instance of the standard simplicial identity in the crossed simplicial group $G_*$ for the pair of indices $\sigma^{-1}(j) < \sigma^{-1}(i)$.
In the second case $\sigma^{-1}(i) < \sigma^{-1}(j)$, the desired identity~\eqref{eq:didj_identity} transforms to the equation
$$d_{\sigma^{-1}(i)}d_{\sigma^{-1}(j)} = d_{\sigma^{-1}(j)-1}d_{\sigma^{-1}(i)},$$
which is again standard.
The remaining simplicial identities are verified in Appendix~\ref{sec:appendix_simplicial_identities}.

\item The case $d_is_j = s_{j-1}d_i$ for $i < j$. Fix $[\sigma, f] \in \Gamma_n$. Then,
$$d_is_j[\sigma, f] = [d_is_j(\sigma), d_{s_j(\sigma)^{-1}(i)}s_{\sigma^{-1}(j)}(f^{-1})^{-1}],$$
while
$$s_{j-1}d_i[\sigma, f] = [s_{j-1} d_i(\sigma), s_{d_i(\sigma)^{-1}(j-1)}d_{\sigma^{-1}(i)}(f^{-1})^{-1}].$$
The first component clearly satisfies the simplicial identity. For the second, we must demonstrate the equality 
\begin{equation}
\label{eq:disj_identity}
    d_{s_j(\sigma)^{-1}(i)}s_{\sigma^{-1}(j)}(f) = s_{d_i(\sigma)^{-1}(j-1)}d_{\sigma^{-1}(i)}(f)
\end{equation}
for arbitrary $f$. This can be verified by considering the two cases $\sigma^{-1}(j) < \sigma^{-1}(i)$ and $\sigma^{-1}(i) < \sigma^{-1}(j)$. 

In the first case $\sigma^{-1}(i) < \sigma^{-1}(j)$, the identity in question~\eqref{eq:disj_identity} transforms to the equation
$$d_{\sigma^{-1}(i)}s_{\sigma^{-1}(i)} = s_{\sigma^{-1}(j)-1}d_{\sigma^{-1}(i)},$$
which is an instance of the standard simplicial identity. In the second case $\sigma^{-1}(j) < \sigma^{-1}(i)$, the identity in question transforms to 
$$d_{\sigma^{-1}(i)+1}s_{\sigma^{-1}(j)} = s_{\sigma^{-1}(j)}d_{\sigma^{-1}(i)},$$
which is again an instance of the standard simplicial identity.

\item The case $s_is_j=s_{j+1}s_i$ for $i \leq j$. Fix $[\sigma, f] \in \Gamma_n$. Then,
$$s_is_j[\sigma, f] = s_i[s_j(\sigma), s_{\sigma^{-1}(j)}(f^{-1})^{-1}] = [s_is_j(\sigma), s_{s_j(\sigma)^{-1}(i)}s_{\sigma^{-1}(j)}(f^{-1})^{-1}],$$
while
$$s_{j+1}s_i[\sigma, f] = s_{j+1}[s_i(\sigma), s_{\sigma^{-1}(i)}(f^{-1})^{-1}] = [s_{j+1}s_i(\sigma), s_{s_i(\sigma)^{-1}(j+1)}s_{\sigma^{-1}(i)}(f^{-1})^{-1}].$$
The first component clearly satisfies the simplicial identity. For the second, we must demonstrate the equality 
\begin{equation}
\label{eq:sisj_identity}
    s_{s_j(\sigma)^{-1}(i)}s_{\sigma^{-1}(j)}(f) = s_{s_i(\sigma)^{-1}(j+1)}s_{\sigma^{-1}(i)}(f)
\end{equation}
for arbitrary $f$. This can be verified by considering the three cases $i = j$, $\sigma^{-1}(j) < \sigma^{-1}(i)$ and $\sigma^{-1}(i) < \sigma^{-1}(j)$. 

In the first case $i = j$, the identity in question~\eqref{eq:sisj_identity} transforms to the equation
$$s_{(s_i\sigma)^{-1}(i)}s_{\sigma^{-1}(i)} = s_{(s_i\sigma)^{-1}(i + 1)}s_{\sigma^{-1}(i)},$$
which is an instance of the standard simplicial identity. In the second case $\sigma^{-1}(i) < \sigma^{-1}(j)$, the identity in question~\eqref{eq:sisj_identity} transforms to the equation $$s_{\sigma^{-1}(i)}s_{\sigma^{-1}(j)} = s_{\sigma^{-1}(j)+1}s_{\sigma^{-1}(i)},$$
which is a standard simplicial identity. In the third case $\sigma^{-1}(j) < \sigma^{-1}(i)$, the identity in question~\eqref{eq:sisj_identity} transforms to the equation
$$s_{\sigma^{-1}(i)+1}s_{\sigma^{-1}(j)} = s_{\sigma^{-1}(j)}s_{\sigma^{-1}(i)},$$
which is an instance of the standard simplicial identity.

\item The case $d_is_j = \id$ for $i \in \{j, j + 1\}$. Fix $[\sigma, f] \in \Gamma_n$. Then,
$$d_is_j[\sigma, f] = [d_is_j(\sigma), d_{s_j(\sigma)^{-1}(i)}s_{\sigma^{-1}(j)}(f^{-1})^{-1}].$$
The first component clearly satisfies the simplicial identity. For the second, we must demonstrate the equality 
\begin{equation}
    \label{eq:disi_identity}
    d_{s_j(\sigma)^{-1}(i)}s_{\sigma^{-1}(j)}(f) = \id.
\end{equation}

This can be verified by considering the two cases $i = j$ and $i = j + 1$. In the first case $i = j$, the identity in question~\eqref{eq:disi_identity} transforms to the equation
$$d_{\sigma^{-1}(i)}s_{\sigma^{-1}(i)} = \id,$$
which is an instance of the same standard simplicial identity. In the second case $i = j + 1$ it transforms to the equation
$$d_{\sigma^{-1}(j)+1}s_{\sigma^{-1}(j)} = \id,$$
which is an instance of the same standard simplicial identity.

\item The case $d_is_j = s_jd_{i-1}$ for $i > j + 1$. Fix $[\sigma, f] \in \Gamma_n$. Then, 
$$d_is_j[\sigma, f] = [d_is_j(\sigma), d_{s_j(\sigma)^{-1}(i)}s_{\sigma^{-1}(j)}(f^{-1})^{-1}],$$
while
$$s_jd_{i-1}[\sigma, f] = [s_{j} d_{i-1}(\sigma), s_{d_{i-1}(\sigma)^{-1}(j)}d_{\sigma^{-1}(i-1)}(f^{-1})^{-1}].$$
The first component clearly satisfies the simplicial identity. For the second, we must demon-
strate the equality 
\begin{equation}
    \label{eq:disj2_identity}
    s_{d_{i-1}(\sigma)^{-1}(j)}d_{\sigma^{-1}(i-1)}(f) = s_{d_{i-1}(\sigma)^{-1}(j)}d_{\sigma^{-1}(i-1)}(f)
\end{equation}
for arbitrary $f$. This can be verified by considering the two cases $\sigma^{-1}(j) < \sigma^{-1}(i-1)$ and $$\sigma^{-1}(j) >  \sigma^{-1}(i-1)$$

In the first case $\sigma^{-1}(j) < \sigma^{-1}(i-1)$, the identity in question~\eqref{eq:disj2_identity} transforms to $$d_{\sigma^{-1}(i-1) + 1}s_{\sigma^{-1}(j)} = s_{\sigma^{-1}(j)}d_{\sigma^{-1}(i-1)},$$
which is an instance of the standard simplicial identity. In the second case $\sigma^{-1}(i-1) < \sigma^{-1}(j)$, the identity in question~\eqref{eq:disj2_identity} transforms to $$d_{\sigma^{-1}(i-1)}s_{\sigma^{-1}(j)} = s_{\sigma^{-1}(j)-1}d_{\sigma^{-1}(i-1)},$$
which is an instance of the standard simplicial identity.

\end{enumerate}

\subsection*{Proof of Theorem \ref{thm:set_operad}}
We proceed with a straightforward verification of the axioms of a shifted operad. The structure of a semi-simplicial module is induced by the crossed simplicial group structure on $G_*$. 

\begin{enumerate}
    \item As the element $\id \in G_0$ we take $1 \in G_0$. Then, $\id \circ_0 \nu = \nu \cdot s_0^\ast(1) = \nu$
    \[\id \circ_0 \nu = \nu \cdot s_0^\ast(1) = \nu \text{ and } \mu \circ_i \id = (1_\ast \boxplus 1 \boxplus 1_\ast) \cdot s_i^0(\mu) = \mu,\]
    since $s_L, s_R$ preserve the identity. Thus, the axioms for $\id$ are satisfied.

    \item The case $(\lambda \circ_i \mu) \circ_{i + j} \nu = \lambda \circ_i(\mu \circ_j \nu)$ for appropriate indices. Fix elements $\lambda \in G_l, \mu \in G_m, \nu \in G_n$ and numbers $0 \leq i \leq l$, $0 \leq j \leq m$. Then, 
    \[(\lambda \circ_i \mu) \circ_{i + j} \nu = (1_{i + j} \boxplus \nu \boxplus 1_{l + m - (i + j)})\cdot  s_{i + j}^n(\underbrace{(1_{i} \boxplus \mu \boxplus 1_{l - i})}_{X} \cdot \underbrace{s_i^m(\lambda)}_{ Y}).\]
    Denote the left factor under $s_{i+j}$ above as $X =1_{i} \boxplus \mu \boxplus 1_{l - i}$ and the second as $Y = s_i^m(\lambda)$. Expand $s_{i + j}(XY) = s_{i + j}(X)s_{X^{-1}(i + j)}(Y)$. Then, 
\[s_{i+j}^n(X) = s_{i + j}^n(s_L^i (s_R^{l - i} (\mu))) = s_L^i s_R^{l - i} s_j^n \mu = 1_i \boxplus s_j^n \mu \boxplus 1_{l - i}\]
by the simplicial identities. Since $X^{-1}(i + j) = i + k$ for some $0 \leq k \leq m$, simplicial identities cause the following equality
\[s_{X^{-1}(i + j)}^n = s_{i + k}^ns_i^m(\lambda) = s_i^{m+n}(\lambda).\] Meanwhile, 
  
    \begin{align*}
        \lambda \circ_i (\mu \circ_j \nu) &= 1_i \boxplus((1_j \boxplus \nu \boxplus 1_{m - j}) \cdot s_j^n(\mu)) \boxplus 1_{l - i}) \cdot s_i^{m + n}(\lambda)  \\ &= (1_{i + j} \boxplus \nu \boxplus 1_{l + m - (i + j)}) \cdot (1_i \boxplus s_j^n \mu \boxplus 1_{l - i}) \cdot s_i^{m + n}(\lambda)
    \end{align*}
    Thus, the desired identity is equivalent to 

    \begin{align*}
        (1_{i + j} \boxplus \nu \boxplus 1_{l + m - (i + j)})\cdot (1_i \boxplus s_j^n \mu \boxplus 1_{l - i}) \cdot s_i^{m + n}(\lambda) \\ =  (1_{i + j} \boxplus \nu \boxplus 1_{l + m - (i + j)}) \cdot (1_i \boxplus s_j^n \mu \boxplus 1_{l - i}) \cdot s_i^{m + n}(\lambda),
    \end{align*}
    which is tautological.

    \item The case $(\lambda \circ_i \mu) \circ_{k + m} \nu  = (\lambda \circ_k \nu ) \circ_i \mu$ for appropriate indices. Fix elements $\lambda \in G_l, \mu \in G_m, \nu \in G_n$ and numbers $0 \leq i < k \leq l$. Then, 
    \[(\lambda \circ_i \mu) \circ_{k + m} \nu = (1_{k + m} \boxplus \nu \boxplus 1_{l -k})\cdot  s_{k + m}^n(\underbrace{(1_i \boxplus \mu \boxplus 1_{l - i} )}_{X} \cdot \underbrace{s_i^m(\lambda)}_{Y}).\]
    
    Denote the left factor under $s_{i+j}$ above as $X =(1_{i} \boxplus \mu \boxplus 1_{l - i}) $ and the second as $Y = s_i^m(\lambda)$. Expand $s_{i + j}(XY) = s_{i + j}(X)s_{X^{-1}(i + j)}(Y)$. Then,
    \[s_{k + m}^n(X) = s_{k + m}^ns_L^i s_R^{l-i} \mu = s_L^i s_{k + m - i}^n s_R^{l - i} \mu = s_L^i s_R^{l-i + n} \mu = (1_i \boxplus \mu \boxplus 1_{l + n - i}).\]
    Since $X(k + m) = k + m$, the following equality is satisfied: 
    \[(\lambda \circ_i \mu) \circ_{k + m} \nu = (1_{k + m} \boxplus \nu \boxplus 1_{m + l - (k + m)}) \cdot (1_i \boxplus \mu \boxplus 1_{l + n - i}) \cdot s_{k + m}^n s_i^m \lambda.\]
    Note that the first two factors commute by the monoidal crossed group axioms. Meanwhile, 
    \[(\lambda \circ_k \nu ) \circ_i \mu  = (1_i \boxplus \mu \boxplus 1_{l + n - i})\cdot s_i^m(\underbrace{(1_k \boxplus \nu \boxplus 1_{l - k})}_{X}\cdot \underbrace{s_k^n(\lambda)}_{Y}).\]
    Denote the left factor under $s_{i}^m$ above as $X =(1_k \boxplus \nu \boxplus 1_{l - k}) $ and the second as $Y = s_k^n(\lambda)$. Note, that since $i < k$, one has $X(i) = i$. Thus, the following expansion is satisfied $s_i(XY) = s_i(X)s_{i}(Y)$. Applying the simplicial identities, one can derive the following: 
    \[s_i(X) = s_i(s_L^k (s_R^{l-k}(\nu))) = s_L^{k + 1} s_R^{l - k} \nu. \]
    Here, the last equality uses the fact that $k > i$. Thus, the desired identity is equivalent to 

    \begin{align*}
      (1_{k + m} \boxplus \nu \boxplus 1_{l - k}) \cdot (1_i \boxplus \mu \boxplus 1_{l + n - i}) \cdot s_{k + m}^n s_i^m (\lambda ) \\ =  (1_i \boxplus \mu \boxplus 1_{l + n - i})\cdot (1_{k + m} \boxplus \nu \boxplus 1_{l - k})\cdot s_i^ms_k^n(\lambda).  
    \end{align*}

    This follows from the fact that the first two multiples commute by the reasons outlined above and the simplicial identities. 

    \item The case $d_{i + j}(\lambda \circ_i \mu) = \lambda \circ_i d_j(\mu )$ for $0 \leq i \leq l, 0 \leq j \leq m$. Fix elements $\lambda \in G_l, \mu \in G_m, \nu \in G_n$ and numbers $0 \leq i \leq l, 0 \leq j \leq m$. Then, 
    \begin{align*}
    d_{i + j}((1_i \boxplus \mu \boxplus 1_{l - i}) \cdot s_i^m(\lambda)) &= d_{i + j}(1_i \boxplus \mu \boxplus 1_{l - i}) \cdot d_{(1_i \boxplus \mu \boxplus 1_{l - i})^{-1}(i + j)} s_i^m(\lambda)\\ &=(1_i \boxplus d_j(\mu) \boxplus 1_{l-i}) \cdot d_{\mu^{-1}(j) + i}s_i^m(\lambda).     
    \end{align*}
    Meanwhile, 
    \[\lambda \circ_i d_j(\mu) = (1_i \boxplus d_j(\mu) \boxplus 1_{l - i}) \cdot s_i^{m-1}(\lambda).\]
    By simplicial identities, 
    \[s_i^{m-1}(\lambda) = d_{\mu^{-1}(j) + i}s_i^m(\lambda),\]
    and thus the required identity holds. 

    \item The cases $d_i(\lambda \circ_k \nu) = d_i(\lambda) \circ_{k} \nu$ and $d_{k + m}(\lambda \circ_i \mu)  = d_k(\lambda) \circ_i \mu$ for $0 \leq i < k \leq l$. These identities follow from straightforward simplicial computations
    \[d_i((1_k \boxplus \nu \boxplus 1_{l-k})\cdot s_k^n(\lambda)) = (1_{k-1} \boxplus \nu \boxplus 1_{l-k}) \cdot d_i s_k^n(\lambda) = (1_{k-1} \boxplus \nu \boxplus 1_{l-k}) \cdot s_k^n(d_i \lambda)\]
    and 
    \[d_{k + m}(( 1_i \boxplus \mu \boxplus 1_{l-i}) \cdot s_i^m(\lambda)) =  (1_i \boxplus \mu \boxplus 1_{l-i - 1}) \cdot d_{k + m}s_i^m(\lambda).\]
\end{enumerate}

\section{Structure map is a Kan fibration}\label{ap:structure_map_Kan}
Let $G_*$ be a (not necessarily symmetric) crossed simplicial group with the structural projection $\pi\colon G_* \to N_*$. In this appendix we provide a sufficient condition for $\pi$ to be a Kan fibration.

\begin{theorem}\label{thm:structure_map_Kan}
    Let $G_*$ be a {\rm(}not necessarily symmetric {\rm(} crossed simplicial group that admits a $\sSet$-section for the structural projection map $\pi\colon G_* \to N_*$. Then $G_* \xrightarrow{\pi} N_*$ is a Kan fibration. 
\end{theorem}

\begin{proof}
    Consider the lifting problem depicted in the diagram below. Denote by $e$ the non-degenerate $n$-simplex of $\Delta^n$. Our goal is to define $\Phi(e) \in G_n$. Let $y_0, \ldots, y_{k-1}, y_{k+1}, \ldots, y_n$ denote the non-degenerate $(n-1)$-simplices of $(\Lambda^n_k)_{n-1}$.
    
    Observe that each element of $G_*$ admits a unique decomposition of the form $p \cdot s$, where $p$ lies in the simplicial subgroup $\ker\{G \xrightarrow{\pi} N\}$ and $s \in \eta(N)$. We seek a solution to the equation 
    \[
        d_r \Phi(e) = y_r
    \]
    of the form $\Phi(e) = p \cdot ((\eta \circ i)(e))$.

\[\begin{tikzcd}
	{\Lambda^n_k} && {G_*} \\
	\\
	{\Delta^n} && {N_*.}
	\arrow["j", from=1-1, to=1-3]
	\arrow["\iota"', from=1-1, to=3-1]
	\arrow["\pi", curve={height=-6pt}, from=1-3, to=3-3]
	\arrow["\Phi", dotted, from=3-1, to=1-3]
	\arrow["i"', from=3-1, to=3-3]
	\arrow["\eta", shift left, curve={height=-6pt}, from=3-3, to=1-3]
\end{tikzcd}\]
    



Now consider the decompositions $y_r = p_r s_r$ for $r = 0, \ldots, n$. Since $p$ lies in the kernel of $\pi$, we have 
    \[
        d_r \Phi(e) = d_r(p) \cdot d_r((\eta \circ i)(e)).
    \]
    Furthermore, because $\pi(d_r(\eta \circ i)(e)) = \pi(y_r) = s_r$ for each $r \neq k$, we seek $p$ satisfying
    \begin{equation}\label{Kancomplex}
        d_0(p) = p_0, \; \ldots, \; d_{k-1}(p) = p_{k-1}, \; d_{k+1}(p) = p_{k+1}, \; \ldots, \; d_n(p) = p_n.
    \end{equation}

    For indices $i < j$, consider
    \[
        d_i y_j = d_i(p_j \cdot s_j) = d_i(p_j) \cdot d_i(s_j),
        \quad \text{and} \quad
        d_{j-1} y_i = d_{j-1}(p_i) \cdot d_{j-1}(s_i).
    \]
    By the uniqueness of the decomposition, it follows that 
    \[
        d_i p_j = d_{j-1} p_i 
        \quad \text{and} \quad 
        d_i s_j = d_{j-1} s_i.
    \]
    Since $\ker\{G_* \xrightarrow{\pi} N\}$ is a simplicial group, it is a Kan complex. Therefore, a solution to equation~\eqref{Kancomplex} exists.
\end{proof}

The principal example of the application of this theorem is provided by the crossed simplicial braid group $B_*$. There is a well-established notion of a \emph{permutation braid}; for a detailed account, see ~\cite{Epstein1992WordPI}. 

\begin{definition}
A braid $\beta \in B_n$ is called a \emph{permutation braid} if it admits a positive representative such that any two strands have at most one crossing.
\end{definition}

Permutation braids are in bijective correspondence with the permutations they induce. We denote by
\(p_n\) the bijection sending a permutation to its associated permutation braid.

It follows directly from the geometric characterization above that the assignment of a permutation to its permutation braid defines a morphism of simplicial sets
\[
p_* \colon S_* \longrightarrow B_*,
\]
which is a section of the structural projection $\pi \colon B_* \to S_*$. Consequently, if $G_*$ is a symmetric simplicial group equipped with a morphism $\psi \colon B_* \to G_*$ satisfying $\pi_G \circ \psi = \pi_B$, then the projection $\pi_G$ also admits a section in the category of simplicial sets.

\[
\begin{tikzcd}
	{B_*} && {G_*} \\
	\\
	& {S_*}
	\arrow["\psi", from=1-1, to=1-3]
	\arrow["\pi"', from=1-1, to=3-2]
	\arrow["\pi", from=1-3, to=3-2]
	\arrow["p"', shift left, curve={height=12pt}, dashed, from=3-2, to=1-1]
\end{tikzcd}
\]

The preceding observations yield numerous examples. In particular, we list the following crossed simplicial groups whose structural projections are Kan fibrations:
\begin{enumerate}
    \item the homotopy braid group $\{hBr_{n+1}\}_{n \geq 0}$;
    \item the welded braid group $\{WBr_{n+1}\}_{n \geq 0}$;
    \item the framed braid groups of surfaces $\{Br^{\mathrm{fr}}_{n+1}(M)\}_{n \geq 0}$, as introduced in ~\cite{Krasauskas1996};
    \item the pullback of the diagram $B_n \to S_n \leftarrow \ZZ/(n+1)$.
\end{enumerate}

\printbibliography

\end{document}